\documentclass[reqno]{amsart}
\usepackage{bbold}
\usepackage{amsmath}
\usepackage{amssymb}
\usepackage{amsfonts}

\newcommand{\abs}[1]{\vert #1 \vert}

\newcommand{\norm}[1]{\left\Vert #1 \right\Vert}

\newcommand{\biggnorm}[1]{\biggl\Vert #1 \biggr\Vert}

\newcommand{\C}{\mathbb{C}}
\newcommand{\N}{\mathbb{N}}

\newcommand{\R}{\mathbb{R}}

\DeclareMathOperator{\im}{Im}
\DeclareMathOperator{\re}{Re}

\newtheorem{theorem}{Theorem}

\newtheorem{lemma}{Lemma}

\theoremstyle{definition}
\newtheorem{definition}{Definition}

\theoremstyle{remark}
\newtheorem{remark}{Remark}

\title[Global existence for nonlinear Dirac equations]{Global existence in the critical space for the Thirring and Gross-Neveu models coupled with the electromagnetic field}

\author[S.~Selberg]{Sigmund Selberg}

\address{Department of Mathematics, University of Bergen, PO Box 7803, 5020 Bergen, Norway}
\email{sigmund.selberg@uib.no}
\date{}

\thanks{The author thanks Jean-Christophe Merle for his hospitality during the author's visit to the University of Vechta, where the main part of the research reported here was carried out.}

\begin{document}

\begin{abstract}
We prove global well-posedness for the coupled Maxwell-Dirac-Thirring-Gross-Neveu equations in one space dimension, with data for the Dirac spinor in the critical space $L^2(\R)$. In particular, we recover earlier results of Candy and Huh for the Thirring and Gross-Neveu models, respectively, without the coupling to the electromagnetic field, but the function spaces we introduce allow for a greatly simplified proof. We also apply our method to prove local well-posedness in $L^2(\R)$ for a quadratic Dirac equation, improving an earlier result of Tesfahun and the author.
\end{abstract}

\maketitle

\section{Introduction}

We consider the following nonlinear Dirac equations on the Minkowski space-time $\R^{1+1}$.
\begin{enumerate}
\item\label{TM} Thirring model:
\[
  (-i\gamma^\mu \partial_\mu + m) \psi = \lambda (\overline \psi \gamma^\mu \psi) \gamma_\mu \psi;
\]

\item\label{GN} Gross-Neveu model (known as the Soler model in higher dimensions)
\[
  (-i\gamma^\mu \partial_\mu + m) \psi = \lambda (\overline \psi \psi) \psi;
\]

\item\label{MD} Maxwell-Dirac equations:
\begin{align*}
  (-i\gamma^\mu \partial_\mu + m) \psi &= \lambda A_\mu \gamma^\mu \psi,
  \\
  \square A_\mu &= - \lambda \overline{\psi} \gamma_\mu \psi, \qquad (\square = \partial^\mu \partial_\mu )
  \\
  \partial^\mu A_\mu &= 0;
\end{align*}

\item\label{MDT} Maxwell-Dirac-Thirring-Gross-Neveu equations:
\begin{align*}
  (-i\gamma^\mu \partial_\mu + m) \psi &= \lambda_1 A_\mu \gamma^\mu \psi + \lambda_2 (\overline \psi \gamma^\mu \psi) \gamma_\mu \psi + \lambda_3 (\overline \psi \psi) \psi,
  \\
  \square A_\mu &= - \lambda_1 \overline{\psi} \gamma_\mu \psi,
  \\
  \partial^\mu A_\mu &= 0;
\end{align*}

\item\label{DKG} Dirac-Klein-Gordon equations:
\begin{align*}
  (-i\gamma^\mu \partial_\mu + m) \psi &= \lambda \phi \psi,
  \\
  (\square + M^2)\phi &= \lambda \overline\psi \psi;
\end{align*}

\item\label{QD} Quadratic Dirac equation:
\[
  (-i\gamma^\mu \partial_\mu + m) \psi = Q(\psi).
\]
\end{enumerate}

Here $\psi = (u,v)^\intercal$ is the Dirac spinor field, which takes values in $\C^2$, the $A_\mu$ are the components of the electromagnetic field and $\phi$ is a real scalar field. The complex conjugate transpose is denoted $\psi^*$  and $\overline \psi = \psi^*\gamma^0$ is the adjoint spinor. The coupling constants $\lambda,\lambda_j$ are assumed to be real, and $m,M \ge 0$ are mass constants. The equations are written in covariant form on $\R^{1+1}$ with coordinates $x^\mu$ ($\mu=0,1$) and metric $(g^{\mu\nu})=\mathrm{diag}(1,-1)$, where $x^0=t$ is time and $x^1=x$ is spatial position, and we write $\partial_\mu = \partial/\partial x^\mu$, so $\partial_0=\partial_t$, $\partial_1=\partial_x$ and $\square = \partial_t^2-\partial_x^2$. The $2 \times 2$ Dirac matrices $\gamma^\mu$ satisfy
\[
  \gamma^\mu \gamma^\nu + \gamma^\nu \gamma^\mu = 2g^{\mu\nu} I \qquad (g^{00}=1, g^{11}=-1, g^{01}=g^{10}=0)
\]
and
\[
  (\gamma^0)^* = \gamma^0, \qquad (\gamma^1)^* = -\gamma^1.
\]
We adopt the representation
\[
  \gamma^0 =
  \begin{pmatrix}
    0 & 1  \\
    1 &0
  \end{pmatrix},
  \qquad
  \gamma^1=
  \begin{pmatrix}
    0 & -1  \\
    1 & 0
  \end{pmatrix}.
\]

The well-posedness of the Cauchy problem for the above models with initial data in Sobolev spaces $H^s(\R) = (1-\partial_x^2)^{-s/2} L^2(\R)$ has been studied by many authors. Our main interest here is the coupled system (iv). Before describing our results, and earlier results, we take a look a the scaling behaviour of the equations, which by standard heuristics indicates a lower bound on the regularity required for well-posedness.

For \eqref{TM} and \eqref{GN}, the scale invariant data space (in the massless case) for the spinor is $H^0=L^2$. For \eqref{MD}, \eqref{DKG} and \eqref{QD}, on the other hand, the scale invariant $H^s$ regularity for the spinor is lower, namely $s=-1/2$ for \eqref{QD} and $s=-1$ for \eqref{MD} and \eqref{DKG} (for the electromagnetic and scalar fields, the scaling regularity is $s=-1/2$).

Thus, based on scaling alone, one would expect that for \eqref{TM} and \eqref{GN}, and hence also \eqref{MDT}, well-posedness fails for spinor data in $H^s$ with $s < 0$, while in $H^0=L^2$ one may hope to have global well-posedness for small-norm data. Moreover, taking into account the fact that the $L^2$ norm of the spinor is a priori conserved (this is the conservation of charge, discussed below), one may speculate that the assumption of small data norm can be removed, and indeed this has been verified for \eqref{TM} and \eqref{GN}: Global $L^2$ well-posedness was proved by Candy \cite{Candy2011} for the Thirring model (see also \cite{Huh2011} for the massless case), and by Huh \cite{Huh2015} for the Gross-Neveu model (see also \cite{Zhang2015}). Earlier global results, with higher regularity, are due to Delgado \cite{Delgado1978} for the Thirring model (and the similar Federbush model), and to Huh \cite{Huh2013} for the Gross-Neveu model.

Delgado \cite{Delgado1978} made the important observation that in the Thirring model, the absolute squares of the spinor components satisfy Dirac type equations with nonlinear terms which are also quadratic in the spinor components. Solving by the method of characteristics and applying Gr\"onwall's lemma, he then obtained an a priori $L^\infty$ bound on the spinor, which can be used to prove global existence in $H^s$ when $s > 1/2$, since then $H^s$ embeds into $L^\infty$. Delgado used the same trick for the Maxwell-Dirac and Dirac-Klein-Gordon equations, for which global existence in $H^1$ had been proved by Chadam \cite{Chadam1973}, and also for the Thirring model coupled to the electromagnetic field. However, Delgado's trick does not work directly for the Gross-Neveu model, since the first-order equations satisfied by the absolute squares of the spinor components then contain quadrilinear terms and not just quadratic ones. Huh \cite{Huh2013} was nevertheless able to salvage Delgado's trick by using integrating factors to effectively get rid of the highest order terms, and using the conservation of charge to bound the integrating factors.

Global $L^2$ (for the spinor) well-posedness for the Dirac-Klein-Gordon equations \eqref{DKG} has been proved by Bournaveas \cite{Bournaveas2000}, but this result relies strongly on the null structure of those equations and does not apply to the Maxwell-Dirac equations \eqref{MD}, which have a weaker null structure. The global result of Bournaveas has been extended to a range of negative Sobolev regularities for the spinor (see \cite{Selberg2007, Tesfahun2009, Candy2013}), and for local well-posedness the range of admissible Sobolev regularity has been completely determined by Machihara, Nakanishi and Tsugawa \cite{Machihara2010} (see also \cite{Machihara2016}).

For the Maxwell-Dirac equations \eqref{MD}, local well-posedness with spinor data in $H^s$, $s > 0$, has been proved by Okamoto \cite{Okamoto2013}. Huh \cite{Huh2010} obtained a global $L^2$ result in the massless case, using the interesting fact that there is then an explicit representation of the solution in terms of the initial data. Bachelot \cite{Bachelot2006} developed in an abstract setting a method to prove global existence and uniqueness for nonlinear and nonlocal hyperbolic systems via an iteration scheme involving the resolvent of the nonlinear part of the equation, and obtained in particular some global existence results for the Maxwell-Dirac and Maxwell-Dirac-Thirring equations with $L^2$ data for the spinor. You and Zhang \cite{Zhang2014} have proved the existence and uniqueness of global weak charge class solutions using Chadam's global existence result and a compactness argument. However, the results of Huh, Bachelot and You-Zhang all have in common that the the electromagnetic potential or its time derivative (or both) are not shown to remain in the space in which the initial data are taken, so they are not well-posedness results. The result that we prove here for \eqref{MDT} applies of course also to Maxwell-Dirac without the Thirring or Gross-Neveu self-interactions, and thus provides the first true global well-posedness result for Maxwell-Dirac with $L^2$ data for the spinor.

For further results on nonlinear Dirac equations in one space dimension, we refer to \cite{Dias1986, Bournaveas2008, Machihara2007, Naumkin2016, Naumkin2016b, Boussaid2016, Contreras2016}. Well-posedness of nonlinear Dirac equations have also been extensively studied in higher dimensions; see \cite{Bournaveas2016, Bejenaru2015, Wang2015, Pecher2014, Ikeda2013, Selberg2011, Grunrock2010, Selberg2010b, Selberg2007} and the references therein.

The proofs of $L^2$ well-posedness for the Thirring and Gross-Neveu models in \cite{Candy2011} and \cite{Huh2015} are based on the method of null-coordinate Sobolev product norms introduced in \cite{Machihara2010}. This creates some technical difficulties since one must localise in both space and time. We choose a different route which greatly simplifies the proof, working instead directly on a time slab $\R \times [0,T]$ with space-time norms motivated by the local form of the conservation of charge, which is discussed next. Our method allows us to easily treat the Thirring and Gross-Neveu models coupled to the electromagnetic field, i.e., the model \eqref{MDT}.

The rest of the paper is organised as follows. In the next section we recall the charge conservation. Our main results are presented in section \ref{MR}. In sections \ref{FS}--\ref{EM} we introduce the function spaces and prove linear and nonlinear estimates. In the remaining sections we then prove the local and global well-posedness results.

\section{Charge conservation}

The models \eqref{TM}--\eqref{DKG} all enjoy the conservation law $\partial_\mu j^\mu = 0$, where $j^\mu = \overline{\psi}\gamma^\mu\psi$ is the Dirac charge density. Since $\psi = (u,v)^\intercal$ we write this as
\[
  \partial_t \rho + \partial_x j = 0,
  \qquad \rho = \abs{\psi}^2 = \abs{u}^2+\abs{v}^2, \qquad j = \overline{\psi} \gamma^1 \psi = \abs{u}^2 - \abs{v}^2.
\]
Integrating this over a time slab $\R \times [0,t]$ gives conservation of total charge
\begin{equation}\label{ChargeConservation}
  \int_{\R} \left( \abs{u(x,t)}^2 + \abs{v(x,t)}^2 \right)\, dx = \int_{\R} \left(\abs{u(x,0)}^2 + \abs{v(x,0)}^2 \right) \, dx
\end{equation}
for a sufficiently regular solution decaying at spatial infinity. On the other hand, integrating over a truncated backward cone $(\R \times [0,t]) \cap \Omega(x_0,t_0)$, where
\[
  \Omega(x_0,t_0) = \left\{ (y,s) \in \R^2 \colon 0 \le s \le t_0, \;\; x_0-(t_0-s) \le y \le x_0 + t_0-s \right\},
\]
gives
\begin{multline}\label{LocalCharge}
  \int_{x_0-t_0+t}^{x_0+t_0-t} \rho(y,t) \, dy
  + \int_0^t (\rho+j)(x_0+t_0-s,s) \, ds
  \\
   + \int_0^t (\rho-j)(x_0-t_0+s,s) \, ds
  = \int_{x_0-t_0}^{x_0+t_0} \rho(y,0) \, dy
\end{multline}
for $0 \le t \le t_0$. Since
\[
  \rho+j=2\abs{u}^2, \qquad \rho-j=2\abs{v}^2
\]
are nonnegative, it follows that the local charge is nonincreasing with increasing time on slices of a backward cone:
\begin{equation}\label{LocalChargeBound}
  \int_a^b \left( \abs{u(x,t)}^2 + \abs{v(x,t)}^2 \right)\, \, dx \le \int_{a-t}^{b+t} \left(\abs{u(x,0)}^2 + \abs{v(x,0)}^2 \right) \, dx
\end{equation}
for any $a<b$ and $t \ge 0$. Moreover, taking $t=t_0$ in \eqref{LocalCharge} gives
\begin{multline}\label{LocalCharge2}
  \int_0^{t} 2\abs{u(x+t-s,s)}^2 \, ds
  + \int_0^{t} 2\abs{v(x-t+s,s)}^2 \, ds
  \\
  = \int_{x-t}^{x+t} \left( \abs{u(y,0)}^2 + \abs{v(y,0)}^2 \right) \, dy.
\end{multline}
This identity was used by Huh \cite{Huh2013, Huh2015} to get global existence for the Gross-Neveu model, and is also a key tool in our proof of global existence. Moreover, it motivates our choice of function spaces. The space-time norms are defined in Section \ref{FS}, but we introduce right away the data space, motivated by the right hand side of \eqref{LocalCharge2}.

\begin{definition} For $T > 0$ set
\[
  \norm{f}_{D(T)} = \sup_{x \in \R} \left( \int_0^T \abs{f(x+2s)}^2 \, ds \right)^{1/2}.
\]
Let $D(T)$ be the completion of $C_c^\infty(\R)$ with respect to this norm.
\end{definition}

Then $L^2(\R) \subset D(T) \subset L^2_{\mathrm{loc}}(\R)$, and the following inequalities hold:
\begin{align}
  \label{f1}
  \norm{f}_{D(T)} &\le 2^{-1/2} \norm{f}_{L^2(\R)},
  \\
  \label{f2}
  \norm{f}_{L^2([a,a+2T])} &\le 2^{1/2} \norm{f}_{D(T)},
  \\
  \label{f3}
  \norm{f}_{L^2([a,a+R])} &\le 2^{1/2} \left( 1 + \frac{R}{2T} \right) \norm{f}_{D(T)}
\end{align}
for any $a \in \R$ and $R > 0$.

We note the following property of the subspace $L^2(\R)$ of $D(T)$.

\begin{lemma}\label{Lemma1}
If $f \in L^2(\R)$, then $\lim_{T \to 0^+} \norm{f}_{D(T)} = 0$.
\end{lemma}

\begin{proof}
The alternative is that there exist $\varepsilon_0 > 0$, $T_n > 0$ and $x_n \in \R$ such that $\lim_{n \to \infty} T_n = 0$ but $\int_0^{T_n} \abs{f(x_n+2t)}^2 \, dt \ge \varepsilon_0$ for all $n \in \N$. Since $f \in L^2(\R)$, the sequence $x_n$ must be bounded, hence it has a subsequence converging to some $x \in \R$. It follows that $\int_{x-\delta}^{x+\delta} \abs{f(y)}^2 \, dy \ge \varepsilon_0$ for all $\delta > 0$, so we have a contradiction.
\end{proof}

\section{Main results}\label{MR}

\subsection{Maxwell-Dirac-Thirring-Gross-Neveu equations} The values of the real coupling constants $\lambda_j$ play no role in our analysis, so to simplify the notation we set them all equal to $1$. In terms of $\psi = (u,v)^\intercal$, the equations then read
\begin{subequations}\label{MD'}
\begin{align}
  \label{MD'a}
(\partial_t + \partial_x)u &= -imv + i (A_0 + A_1) u + 2i \abs{v}^2 u + 2i \re(u\overline v) v,
  \\
  \label{MD'b}
  (\partial_t - \partial_x)v &= -imu + i (A_0 - A_1) v + 2i \abs{u}^2 v + 2i \re(u\overline v) u,
  \\
  \label{MD'c}
  (\partial_t^2-\partial_x^2) A_0 &= - (\abs{u}^2 + \abs{v}^2),
  \\
  \label{MD'd}
  (\partial_t^2-\partial_x^2) A_1 &= \abs{u}^2 - \abs{v}^2,
  \\
  \label{MD'e}
  \partial_t A_0 &= \partial_x A_1.
\end{align}
\end{subequations}

The last equation, the Lorenz gauge condition, imposes an obvious constraint on the initial data. An additional data constraint (which appears to have been ignored in earlier treatments of the Maxwell-Dirac equations in one space dimension, with the exception of Okamoto's paper \cite{Okamoto2013}) comes from the relation
\[
  \partial_t A_1 = \partial_x A_0 - E
\]
defining the electric field $E$. Thus, the initial conditions for \eqref{MD'} take the form
\begin{subequations}\label{MDdata}
\begin{alignat}{2}
   \label{MDdata1}
   u(x,0) &= f(x),& \qquad v(x,0) &= g(x),
   \\
   \label{MDdata2}
   A_0(x,0) &= a_0(x),& \qquad A_1(x,0) &= a_1(x), 
   \\
   \label{MDdata3}
   \partial_t A_0(x,0) &= \frac{d a_1}{dx} (x),& \qquad \partial_t A_1(x,0) &= \frac{d a_0}{dx} (x) - E_0(x),
\end{alignat}
\end{subequations}
where $E_0$ denotes the initial value of the electric field.

From \eqref{MD'c}--\eqref{MD'e} one obtains the Gauss law
\[
  \partial_x E = \abs{u}^2 + \abs{v}^2,
\]
implying
\[
  E(x,0) = E_0(x) = \kappa + \int_0^x (\abs{f(y)}^2 + \abs{g(y)}^2) \, dy
\]
for some constant $\kappa \in \R$. So if $f,g \in L^2(\R)$, then $E_0$ is a bounded and continuous function (in fact, absolutely continuous and of bounded variation), determined by $f$ and $g$ up to a constant. However, our approach will be to first prove local existence of \eqref{MD'a}--\eqref{MD'd} with data \eqref{MDdata} for a completely freely chosen field $E_0 \in BC(\R;\R)$. Then we show that if this field actually satisfies the Gauss law, then the solution satisfies the Lorenz gauge condition \eqref{MD'e}.

\begin{remark}
The important parts of the data are $f$, $g$ and $E_0$. The fields $a_0$ and $a_1$, on the other hand, are rather uninteresting. In fact, by a gauge transformation one can see that the effect of varying $a_0$ and $a_1$ is simply to multiply $(u,v)$ by a phase factor. Indeed, let $\chi(x,t)$ solve
\[
  (\partial_t^2 - \partial_x^2)\chi = 0, \qquad \chi(0,x) = \chi_0(x),
  \qquad \partial_t \chi(0,x) = \chi_1(x)
\]
for given real valued $\chi_0(x)$ and $\chi_1(x)$. Then \eqref{MD'} is invariant under the gauge transformation
\[
  (u,v,A_0,A_1) \to (u',v',A_0',A_1') = (e^{i\chi}u,e^{i\chi}v,A_0-\partial_t \chi,A_1-\partial_x \chi).
\]
Now, we can make the initial values of $A_0'$ and $A_1'$ equal any choice of fields $a_0'$ and $a_1'$, by choosing
\[
  \chi_0(x) = \int_0^x (a_1-a_1')(y) \, dy, \qquad \chi_1(x) = a_0(x) - a_0'(x).
\]
Then we have moreover $\partial_t A_0'(x,0) = \frac{d}{dx} a_1'(x)$ and $\partial_t A_1'(x,0) = \frac{d}{dx} a_0'(x) - E_0(x)$, so the initial conditions \eqref{MDdata3} are gauge invariant. Note that if the $a$'s are continuous, then $\chi \in C^1(\R \times \R)$, so the above formal considerations are justified.
\end{remark}

We now state our local and global results for \eqref{MD'}.

\begin{theorem}[Local well-posedness for Maxwell-Dirac-Thirring-Gross-Neveu]\label{Thm1} There exists $\varepsilon_0 > 0$ such that for any time $T > 0$ and for any initial data $f,g \in D(T)$ and $a_0,a_1,E_0 \in BC(\R;\R)$ satisfying
\begin{equation}\label{T2}
  \norm{f}_{D(T)}^2 + \norm{g}_{D(T)}^2 \le \varepsilon_0
\end{equation}
and
\begin{equation}\label{T1}
  T\left( m + \norm{a_0}_{L^\infty} + \norm{a_1}_{L^\infty}\right) +  T^2\norm{E_0}_{L^\infty} \le \varepsilon_0,
\end{equation}
the equations \eqref{MD'a}--\eqref{MD'd} have a unique solution $(u,v,A_0,A_1)$ on $\R \times (0,T)$ satisfying the initial conditions \eqref{MDdata}, the regularity conditions
\begin{subequations}\label{Regularity}
\begin{align}
  \label{RegularityA}
  u,v &\in C([0,T];D(T)),
  \\
  \label{RegularityC}
  A_0,A_1 &\in C([0,T];BC(\R;\R))
\end{align}
\end{subequations}
and the properties
\begin{enumerate}
  \item\label{P1} The local charge conservation \eqref{LocalCharge} holds, hence also \eqref{LocalChargeBound} and \eqref{LocalCharge2}.
  \item\label{P2} There exist $p,q \in D(T)$ such that $\abs{u(x,t)} \le p(x-t)$ and $\abs{v(x,t)} \le q(x+t)$.
\end{enumerate}
Moreover, the electric field $E := \partial_x A_0 - \partial_t A_1$ belongs to $C([0,T];BC(\R;\R))$, the solution depends continuously on the data as a map into the class \eqref{Regularity}, and higher regularity persists, so in particular, the solution is a limit of smooth solutions.

If we assume additionally that $f,g \in L^2(\R)$, then $u,v \in C([0,T];L^2(\R))$ and depend continuously on the data as maps into that space, and the conservation of total charge \eqref{ChargeConservation} holds.

So far, $E_0$ was a freely chosen field in $BC(\R,\R)$. However, if we assume in addition that it satisfies the Gauss law
\begin{equation}\label{GaussLaw}
  \frac{d}{dx} E_0 = \abs{f}^2 + \abs{g}^2,
\end{equation}
then the Lorenz gauge condition \eqref{MD'e} is satisfied in $\R \times (0,T)$.
\end{theorem}

\begin{remark}
If $f,g \in L^2(\R)$, then the smallness condition \eqref{T2} is satisfied for $T > 0$ small enough, by Lemma \ref{Lemma1}.
\end{remark}

\begin{remark}\label{Drem}
When we say that $(u,v,A_0,A_1)$ is a solution on $\R \times (0,T)$, we mean in the sense of distributions. In fact, all the nonlinear terms on the right hand sides of \eqref{MD'a}--\eqref{MD'd} are locally integrable on $\R \times [0,T]$, so they are well-defined as distributions. This is obvious for the quadratic terms, but not so obvious for the cubic terms. But by the properties \eqref{P1} and \eqref{P2} in Theorem \ref{Thm1},
\begin{align*}
  \int_a^{a+T} \!\!\! \int_0^T \abs{v}^2 \abs{u}(x,t) \, dt \, dx 
  &\le
  \int_{a-T}^{a+T} \!\!\! \int_0^T \abs{v(y+t,t)}^2 p(y) \, dt \, dy
  \\
  &\le
  2\sqrt{T} \norm{p}_{D(T)} \left( \sup_{y} \int_0^T \abs{v(y+t,t)}^2 \, dt \right) < \infty,
\end{align*}
and similarly for $\abs{u}^2 \abs{v}$.
\end{remark}

\begin{theorem}[Global existence for Maxwell-Dirac-Thirring-Gross-Neveu]\label{Thm2} For any data $f,g \in L^2(\R)$ and $a_0,a_1,E_0 \in BC(\R;\R)$ satisfying the Gauss law \eqref{GaussLaw}, the solution from Theorem \ref{Thm1} extends globally in time, so we have $u,v \in C(\R;L^2(\R))$ and $A_0,A_1,E \in C(\R;BC(\R;\R))$.
\end{theorem}

\subsection{Quadratic Dirac equations}

Here we consider equations of the form
\begin{subequations}\label{QD'}
\begin{align}
  \label{QD'a}
  (\partial_t + \partial_x)u &= -imv + c_1 \abs{v}^2 + c_2 u v,
  \\
  \label{QD'b}
  (\partial_t - \partial_x)v &= -imu + c_3 \abs{u}^2 + c_4 u v,
\end{align}
\end{subequations}
where the $c_j$ are complex constants. Then we have the following local well-posedness result.

\begin{theorem}\label{Thm3} There exists $\varepsilon_0 > 0$ such that for any time $T > 0$ and any initial data $f,g \in L^2(\R)$ satisfying
\[
  \sqrt{T} \bigl( m + \norm{f}_{L^2} + \norm{g}_{L^2}\bigr) \le \varepsilon_0
\]
the equations \eqref{QD'} have a solution in $C([0,T];L^2(\R))$ with data $(f,g)$. The solution satisfies the properties \eqref{P1} and \eqref{P2} from Theorem \ref{Thm1} and is unique in this class. The solution depends continuously on the data and higher regularity persists.
\end{theorem}

If $c_1=c_3=0$, so there are only terms of the type $uv$, one can do better; in fact, local well-posedness holds almost all the way down to the scale invariant $H^s$ Sobolev regularity $s=-1/2$, as shown by Machihara, Nakanishi and Tsugawa \cite[Theorem 1.5]{Machihara2010}. For $c_2=c_4=0$, local well-posedness in $H^s$ for $s > 1/4$ has been proved by Machihara \cite{Machihara2005} and $s > 0$ by Tesfahun and the author \cite{Selberg2010}, but the case $s=0$ appears to have remained an open problem which we solve.

\section{Function spaces}\label{FS}

We start with the spaces for the Dirac spinor, which we write in component form as $\psi= (u,v)^\intercal$. Then by our choice of Dirac matrices, the initial value problem
\[
  -i\gamma^\mu \partial_\mu \psi = (F,G)^\intercal, \qquad \psi(x,0) = (f(x),g(x))^\intercal
\]
reads
\begin{subequations}\label{IVP}
\begin{alignat}{2}
  \label{IVPa}
  (\partial_t + \partial_x)u &= iG,& \qquad u(x,0) &= f(x),
  \\
  \label{IVPb}
  (\partial_t - \partial_x)v &= iF,& \qquad v(x,0) &= g(x).
\end{alignat}
\end{subequations}
Integrating along characteristics we have
\begin{align*}
  u(x,t) &= f(x-t) +  i\int_0^t G(x-t+s,s) \, ds,
  \\
  v(x,t) &= g(x+t) +  i\int_0^t F(x+t-s,s) \, ds,
\end{align*}
hence
\begin{subequations}\label{IVPsol}
\begin{align}
  \label{IVPsola}
  \abs{u(x,t)} &\le \abs{f(x-t)} +  \int_0^T \abs{G(x-t+s,s)} \, ds,
  \\
  \label{IVPsolb}
  \abs{v(x,t)} &\le \abs{g(x+t)} +  \int_0^T \abs{F(x+t-s,s)} \, ds
\end{align}
\end{subequations}
for $0 \le t \le T$, 

We introduce the following space-time norms for $u$ and $v$ on a time-slab $\R \times [0,T]$:
\begin{align*}
  \norm{u}_{X_+(T)} &= \sup_{x \in \R} \left( \int_0^T \abs{u(x-t,t)}^2 \, dt \right)^{1/2},
  \\
  \norm{v}_{X_-(T)} &= \sup_{x \in \R} \left( \int_0^T \abs{v(x+t,t)}^2 \, dt \right)^{1/2},
\end{align*}
\begin{align*}
  \norm{u}_{\mathcal X_+(T)} &= \inf \left\{ \norm{p}_{D(T)} \colon p \in D(T), \; \text{$\abs{u(x,t)} \le p(x-t)$ in $\R \times [0,T]$} \right\},
  \\
  \norm{v}_{\mathcal X_-(T)} &= \inf \left\{ \norm{q}_{D(T)} \colon q \in D(T), \; \text{$\abs{v(x,t)} \le q(x+t)$ in $\R \times [0,T]$} \right\}
\end{align*}
and
\begin{align*}
  \norm{u}_{Y_+(T)} = \sup_{0 \le t \le T} \norm{u(\cdot,t)}_{D(T)} + \norm{u}_{X_+(T)} + \norm{u}_{\mathcal X_+(T)},
  \\
  \norm{v}_{Y_-(T)} = \sup_{0 \le t \le T} \norm{v(\cdot,t)}_{D(T)} + \norm{v}_{X_-(T)} + \norm{v}_{\mathcal X_-(T)}.
\end{align*}
The forcing terms $G$ and $F$ will be placed in the norms
\begin{align*}
  \norm{G}_{N_+(T)} &= \biggnorm{ \int_0^T \abs{G(\cdot + s,s)} \, ds }_{D(T)},
  \\
  \norm{F}_{N_-(T)} &= \biggnorm{ \int_0^T \abs{F(\cdot - s,s)} \, ds }_{D(T)}.
\end{align*}

\begin{remark}
The $X_\pm(T)$ norms are motivated by the local charge conservation \eqref{LocalCharge2} (see also the estimate in Remark \ref{Drem}). To motivate the $\mathcal X_\pm(T)$ norms, observe that by \eqref{IVPsol} we have $\abs{u(x,t)} \le p(x-t)$ and $\abs{v(x,t)} \le q(x+t)$, where $p$ and $q$ belong to $D(T)$ if $f,g \in D(T)$ and $\norm{G}_{N_+(T)}, \norm{F}_{N_-(T)} < \infty$.
\end{remark}

\begin{definition} Let $Y_\pm(T)$ and $N_\pm(T)$ be the completions of $C_c^\infty(\R \times [0,T])$ with respect to the norms $\norm{\cdot}_{Y_\pm(T)}$ and $\norm{\cdot}_{N_\pm(T)}$, respectively.
\end{definition}

Note that
\[
  Y_\pm(T) \subset C([0,T];D(T)).
\]

For $u(x,t) = f(x-t)$ and $v(x,t) = g(x+t)$, the free parts of the solution of \eqref{IVP}, we have the remarkable identities
\begin{equation}\label{KeyIdentity1}
\begin{aligned}
  \norm{f(x-t)}_{X_+(T)} &= \norm{f(x-t)}_{\mathcal X_+(T)} = \norm{f}_{D(T)},
  \\
  \norm{g(x+t)}_{X_-(T)} &= \norm{g(x+t)}_{\mathcal X_-(T)} = \norm{g}_{D(T)},
\end{aligned}
\end{equation}
and for a constant function $u(x,t) = c$,
\begin{equation}\label{KeyIdentity2}
  \norm{c}_{X_\pm(T)} = \norm{c}_{\mathcal X_\pm(T)} = \norm{c}_{D(T)} = c\sqrt{T}.
\end{equation}

For the solution of the linear initial value problem \eqref{IVP} we then have the following key estimates.

\begin{lemma}\label{Lemma2}
Let $T > 0$, $f,g \in D(T)$ and $(G,F) \in N_+(T) \times N_-(T)$. Then \eqref{IVP} has a unique solution $(u,v) \in Y_+(T) \times Y_-(T)$. The solution satisfies the estimates
\begin{align*}
  \norm{u}_{Y_+(T)} &\le 3 \norm{f}_{D(T)} + 3 \norm{G}_{N_+(T)},
  \\
  \norm{v}_{Y_-(T)} &\le 3 \norm{g}_{D(T)} + 3 \norm{F}_{N_-(T)}.
\end{align*}
\end{lemma}

\begin{proof}
By density we may assume that $f,g \in C_c^\infty(\R)$ and $F,G \in C_c^\infty(\R \times [0,T])$. Applying the identities \eqref{KeyIdentity1} to \eqref{IVPsol} and using the translation invariance of the $D(T)$ norm, we immediately obtain the claimed estimates.
\end{proof}

For later use we also note that by Minkowski's integral inequality,
\begin{equation}\label{Nineq}
 \norm{u}_{N_\pm(T)} \le \int_0^T \norm{u(\cdot,s)}_{D(T)} \, ds \le T\norm{u}_{Y_{\pm'}(T)}
\end{equation}
for any combination of signs $\pm$ and $\pm'$.

\section{Estimates for the Dirac spinor}\label{DS}

To estimate linear, quadratic and cubic terms in $u$ and $v$ in the $N_\pm(T)$ norms, we use the following key result.

\begin{lemma}\label{Lemma3} Let $u,u' \in Y_+(T)$ and $v,v' \in \times Y_-(T)$. We have the estimates
\begin{align*}
  \norm{v v' u}_{N_+(T)} &\le \norm{v}_{X_-(T)} \norm{v'}_{X_-(T)} \norm{u}_{\mathcal X_+(T)},
  \\
  \norm{u u' v}_{N_-(T)} &\le \norm{u}_{X_+(T)} \norm{u'}_{X_+(T)} \norm{v}_{\mathcal X_-(T)}
\end{align*}
and
\begin{align*}
  \norm{v v'}_{N_+(T)} &\le \sqrt{T} \norm{v}_{X_-(T)} \norm{v'}_{X_-(T)},
  \\
  \norm{v u}_{N_+(T)} &\le \sqrt{T} \norm{v}_{X_-(T)} \norm{u}_{\mathcal X_+(T)},
  \\
  \norm{u u'}_{N_-(T)} &\le \sqrt{T} \norm{u}_{X_+(T)} \norm{u'}_{X_+(T)},
  \\
  \norm{u v}_{N_-(T)} &\le \sqrt{T} \norm{u}_{X_+(T)} \norm{v}_{\mathcal X_-(T)}
\end{align*}
and
\begin{align*}
  \norm{v}_{N_+(T)} &\le T \norm{v}_{X_-(T)},
  \\
  \norm{u}_{N_-(T)} &\le T \norm{u}_{X_+(T)}.
\end{align*}
\end{lemma}

\begin{proof}
We only prove the cubic estimates; the others follow by the identities \eqref{KeyIdentity2} with $c=1$. Consider the first cubic estimate. The left side is bounded by
\[
  \norm{ \int_0^T \abs{v} \abs{v'} \abs{u}(\cdot+s,s) \, ds }_{D(T)}
\]
Assuming $\abs{u(x,t)} \le p(x-t)$ we bound this by
\[
  \norm{ \left(\int_0^T \abs{v}\abs{v'}(\cdot+s,s) \, ds \right) p(\cdot) }_{D(T)}
  \le
  \norm{v}_{X_-(T)} \norm{v'}_{X_-(T)} \norm{p}_{D(T)},
\]
where we applied H\"older's inequality with respect to $s$. Taking the infimum over $p \in D(T)$ such that $\abs{u(x,t)} \le p(x-t)$ now yields the desired estimate. The proof of the other cubic estimate is similar.
\end{proof}

\begin{remark}
In the cubic estimates, the number of plus signs equals the number of minus signs, and in the quadratic estimates both signs must occur. This reflects the fact that the estimates are of null form type. We are not able to estimate $uu'$, $uu'v$ or $uu'u''$ in $N_+(T)$, nor $vv'$, $vv'u$ or $vv'v'''$ in $N_-(T)$. For the linear estimates, on the other hand, the signs are not really important, in view of \eqref{Nineq}.
\end{remark}

\section{Estimates for the electromagnetic field}\label{EM}

Applying D'Alembert's formula to the wave equations \eqref{MD'c} and \eqref{MD'd} with initial data \eqref{MDdata2} and \eqref{MDdata3} we have
\begin{multline}\label{A0}
  A_0(x,t) = \frac{a_0(x+t) + a_0(x-t)}{2} + \frac{a_1(x+t) - a_1(x-t)}{2}
  \\
  - \frac12 \int_0^t \int_{x-(t-s)}^{x+t-s} (\abs{u}^2+\abs{v}^2)(y,s) \, dy \, ds
\end{multline}
and
\begin{multline}\label{A1}
  A_1(x,t) = \frac{a_0(x+t) - a_0(x-t)}{2} + \frac{a_1(x+t) + a_1(x-t)}{2}
  - \frac12 \int_{x-t}^{x+t} E_0(y) \, dy
  \\
  + \frac12 \int_0^t \int_{x-(t-s)}^{x+t-s} (\abs{u}^2-\abs{v}^2)(y,s) \, dy \, ds.
\end{multline}
Thus
\begin{align}
  \label{Aplus}
  A_0 + A_1 &= A_+^{\mathrm{free}} - W\left(\abs{v}^2\right),
  \\
  \label{Aminus}
  A_0 - A_1 &= A_-^{\mathrm{free}} - W\left(\abs{u}^2\right),
\end{align}
where
\begin{align*}
  A_+^{\mathrm{free}}(x,t) &= a_0(x+t) + a_1(x+t) - \frac12 \int_{x-t}^{x+t} E_0(y) \, dy,
  \\
  A_-^{\mathrm{free}}(x,t) &= a_0(x-t) - a_1(x-t) + \frac12 \int_{x-t}^{x+t} E_0(y) \, dy
\end{align*}
and the operator $W$ is defined by
\[
  WF(x,t) =
  \int_0^t \int_{x-(t-s)}^{x+t-s} F(y,s) \, dy \, ds.
\]

We note the following estimates for the above operators.

\begin{lemma}\label{Lemma4}
Let $T > 0$, $(a_0,a_1,E_0) \in BC(\R;\R)$ and $u,v \in C([0,T];D(T))$. Then $A_\pm^{\mathrm{free}}$ and $W(uv)$ belong to $C([0,T];BC(\R;\R))$ and we have the estimates
\begin{align*}
  \norm{A_\pm^{\mathrm{free}}}_{L^\infty(\R \times [0,T])} &\le \norm{a_0}_{L^\infty} + \norm{a_1}_{L^\infty} + T \norm{E_0}_{L^\infty},
  \\
  \norm{W(uv)}_{L^\infty(\R \times [0,T])} &\le 2 \int_0^T \norm{u(\cdot,s)}_{D(T)}\norm{v(\cdot,s)}_{D(T)} \, ds.
\end{align*}
Moreover, higher regularity persists. That is, if $(a_0,a_1,E_0) \in BC^N(\R;\R)$ for some $N \in \N$, and $\partial_t^j \partial_x^k(u,v) \in C([0,T];D(T))$ for $j = 0,1$ and $k = 0,1,\dots,N-j$, then it follows that $\partial_t^j \partial_x^k A_\pm^{\mathrm{free}}$ and $\partial_t^j \partial_x^k W(uv)$ belong to $C([0,T];BC(\R;\R))$ for such $j,k$, and we have estimates analogous to the above.
\end{lemma}

We omit the easy proof, but remark that for the higher regularity one needs the identities
\begin{align*}
  \partial_x^k W(F) &= W(\partial_x^k F),
  \\
  \partial_t\partial_x^k W(F) &= W(\partial_t\partial_x^k F) + \frac12 \int_{x-t}^{x+t} \partial_x^k F(y,0) \, dy.
\end{align*}

\begin{remark}
The reason that we limit to one time derivative in the higher regularity statement is that for $j \ge 2$,
\begin{multline*}
  \partial_t^j\partial_x^k W(F) = W(\partial_t^j\partial_x^k F) + \frac12 \left( \partial_t^{j-2}\partial_x^k F(x+t,0) + \partial_t^{j-2}\partial_x^k F(x-t,0) \right)
  \\
  + \frac12 \int_{x-t}^{x+t} \partial_t^{j-1}\partial_x^k F(y,0) \, dy,
\end{multline*}
and terms like $\partial_t^{j-2}\partial_x^k (fg)$ cannot be estimated in $L^\infty$ in terms of $\norm{f}_{D(T)}$ and $\norm{g}_{D(T)}$. However, higher time derivatives can be recovered from the equations \eqref{MD'} once we have $C^\infty$ regularity in space and $C^1$ in time.
\end{remark}

\section{Proof of Theorem \ref{Thm1}}\label{LWP}

Inserting \eqref{Aplus} and \eqref{Aminus} into \eqref{MD'a} and \eqref{MD'b} yields the following nonlinear and nonlocal Dirac equations:
\begin{align*}
  (\partial_t + \partial_x)u &= -imv + i \left( A_+^{\mathrm{free}} - W\left(\abs{v}^2\right) \right) u + 2i \abs{v}^2 u + 2i \re(u\overline v) v,
  \\
  (\partial_t - \partial_x)v &= -imu + i \left( A_-^{\mathrm{free}} - W\left(\abs{u}^2\right) \right) v + 2i \abs{u}^2 v + 2i \re(u\overline v) u.
\end{align*}

Iterating in $Y_+(T) \times Y_-(T)$, we get a unique local solution $(u,v)$ in that space if $T > 0$ satisfies \eqref{T2} and \eqref{T1}. This follows by a standard iteration argument (we omit the details) using Lemma \ref{Lemma2}, the linear and trilinear estimates from Lemma \ref{Lemma3}, the estimate (here we use \eqref{Nineq})
\[
  \norm{Au}_{N_\pm(T)} \le \norm{A}_{L^\infty(\R \times [0,T])} \norm{u}_{N_\pm(T)}
  \le \norm{A}_{L^\infty(\R \times [0,T])} \norm{u}_{Y_{\pm'}(T)},
\]
and the estimates in Lemma \ref{Lemma4}. Note also that Lemma \ref{Lemma4} implies the regularity \eqref{RegularityC} for the $A_\mu$. Then $(u,v,A_0,A_1)$ is a solution of \eqref{MD'a}--\eqref{MD'd}.

The iteration argument also gives continuous dependence on the data, and higher regularity persists. Thus, $(u,v,A_0,A_1)$ is a limit, in the regularity class \eqref{Regularity}, of $C^\infty$ solutions corresponding to an approximating sequence of $C^\infty$ data. In particular, it follows that the local charge conservation \eqref{LocalCharge} holds, hence also \eqref{LocalChargeBound} and \eqref{LocalCharge2}.

The electric field $E = \partial_x A_0 - \partial_t A_1$ is calculated directly from \eqref{A0} and \eqref{A1},
\begin{multline}\label{E}
  E(x,t) =
  - \int_0^t \abs{u(x+t-s,s)}^2 \, ds
  + \int_0^t \abs{v(x-t+s,s)}^2 \, ds
  \\
  + \frac12 \left( E_0(x+t) + E_0(x-t) \right).
\end{multline}
This shows that $E$ is continuous and enjoys the bound
\[
  \norm{E}_{L^\infty(\R \times [0,T])} \le \norm{E_0}_{L^\infty} + \norm{u}_{X_+(T)}^2 + \norm{v}_{X_-(T)}^2.
\]

From \eqref{A0} and \eqref{A1} we also calculate
\begin{multline*}
  \partial_t A_0(x,t) - \partial_x A_1(x,t)
  =
  - \int_0^t \abs{u(x+t-s,s)}^2 \, ds - \int_0^t \abs{v(x-t+s,s)}^2 \, ds
  \\
  + \frac12 \left( E_0(x+t) - E_0(x-t) \right),
\end{multline*}
so from the local form of conservation of charge \eqref{LocalCharge2} it follows that the Lorenz gauge condition is satisfied if the Gauss law \eqref{GaussLaw} is satisfied initially.

Finally, it remains to prove that if $f,g \in L^2(\R)$, then $u,v \in C([0,T];L^2(\R))$ and depend continuously on the data as maps into that space. To this end, we use the local charge bound \eqref{LocalChargeBound}. Considering two solutions $(u,v)$ and $(u',v')$, we estimate, for $0 \le s,t \le T \ll R$,
\begin{align*}
  &\norm{u(\cdot,s)-u'(\cdot,t)}_{L^2}
  \\
  &\quad\le
  \norm{u(\cdot,s)-u'(\cdot,t)}_{L^2(\abs{x} \le R)}
  +
  \norm{u(\cdot,s)}_{L^2(\abs{x} \ge R)}
  +
  \norm{u'(\cdot,t)}_{L^2(\abs{x} \ge R)},
  \\
  &\quad\le
  2\frac{R}{T}\norm{u(\cdot,s)-u'(\cdot,t)}_{D(T)}
  +
  \norm{f}_{L^2(\abs{x} \ge R-T)} + \norm{g}_{L^2(\abs{x} \ge R-T)}
  \\
  &\quad \qquad +
  \norm{f'}_{L^2(\abs{x} \ge R-T)} + \norm{g'}_{L^2(\abs{x} \ge R-T)}
    \\
  &\quad\le
  2\frac{R}{T}\norm{u(\cdot,s)-u'(\cdot,t)}_{D(T)}
  +
  2\left( \norm{f}_{L^2(\abs{x} \ge R-T)} + \norm{g}_{L^2(\abs{x} \ge R-T)} \right)
  \\
  &\quad \qquad +
  \norm{f-f'}_{L^2} + \norm{g-g'}_{L^2},
\end{align*}
where we used \eqref{LocalChargeBound} and \eqref{f3} to get the second inequality.

Applying this with $u'=0$ shows that $u(\cdot,t)$ is in $L^2$. Applying it with $u'=u$ shows that $t \mapsto u(\cdot,t)$ is continuous into $L^2$. Finally, applying it with $u' = u_n$, where $(u_n,v_n)$ is a solution that converges to $(u,v)$ in $C([0,T];D(T))$ and such that $(f_n,g_n) \to (f,g)$ in $L^2$, we conclude that the convergence $u_n \to u$ actually holds in $C([0,T];L^2(\R))$.

\section{Proof of Theorem \ref{Thm3}}\label{LWPQD}

This is similar to but easier than the proof of Theorem \ref{Thm1}, and uses the quadratic estimates in Lemma \ref{Lemma3} instead of the cubic ones. We omit the details.

\section{Proof of Theorem \ref{Thm2}}\label{GWP}

We prove existence on the time interval $[0,\infty)$; by the reflection $(x,t) \to (-x,-t)$ one can obtain the same result for negative times.

By the local existence result, Theorem \ref{Thm1}, it suffices to exhibit a priori bounds on $u(\cdot,t)$ and $v(\cdot,t)$ in $D(T)$, and on $A_\mu(\cdot,t)$ and $E(\cdot,t)$ in $L^\infty$.

For $A_0$ and $A_1$ we have, by the formulas \eqref{A0} and \eqref{A1}, and using the local charge bound \eqref{LocalChargeBound},
\begin{equation}\label{Abound}
  \norm{A_\mu(\cdot,t)}_{L^\infty} \le \norm{a_0}_{L^\infty} + \norm{a_1}_{L^\infty} + t \norm{E_0}_{L^\infty} + \frac{t}{2} \left( \norm{f}_{L^2}^2 + \norm{g}_{L^2}^2 \right).
\end{equation}
From the formula \eqref{E} for $E$ we get, invoking the local charge identity \eqref{LocalCharge2},
\begin{equation}\label{Ebound}
  \norm{E(\cdot,t)}_{L^\infty} \le  \norm{E_0}_{L^\infty} + \frac12 \left( \norm{f}_{L^2}^2 + \norm{g}_{L^2}^2 \right).
\end{equation}
To bound
\[
  \norm{u(\cdot,t)}_{D(T)}^2 + \norm{v(\cdot,t)}_{D(T)}^2
\]
we use Delgado's trick \cite{Delgado1978}, or rather a refinement of it due to Huh \cite{Huh2013}. The starting point is the observation that
\begin{equation}\label{DelgadoFails}
\begin{aligned}
  (\partial_t + \partial_x) \abs{u}^2 &= - 2m \im (u\overline v) + 4\re(u\overline v)\im(u\overline v),
  \\
  (\partial_t - \partial_x) \abs{v}^2 &= 2m \im (u\overline v) - 4\re(u\overline v)\im(u\overline v),
\end{aligned}
\end{equation}
hence
\begin{align*}
  (\partial_t + \partial_x) \abs{u}^2 &\le 2m \abs{u}\abs{v} + 4\abs{u}^2\abs{v}^2,
  \\
  (\partial_t - \partial_x) \abs{v}^2 &\le 2m \abs{u}\abs{v} + 4\abs{u}^2\abs{v}^2.
\end{align*}
Multiplying by integrating factors $e^{-\phi_+}$ and $e^{-\phi_-}$, where
\begin{align*}
  \phi_+(x,t) &= \int_0^t 4\abs{v(x-t+s,s)}^2 \, ds,
  \\
  \phi_-(x,t) &= \int_0^t 4\abs{u(x+t-s,s)}^2 \, ds,
\end{align*}
we get
\begin{align*}
  (\partial_t + \partial_x) (e^{-\phi_+}\abs{u}^2) &\le 2m \abs{u}\abs{v} e^{-\phi_+},
  \\
  (\partial_t - \partial_x) (e^{-\phi_-}\abs{v}^2) &\le 2m \abs{u}\abs{v} e^{-\phi_-}.
\end{align*}
By the local charge conservation \eqref{LocalCharge2},
\[
  \norm{\phi_\pm}_{L^\infty} \le 2M,
\]
where
\[
  M=\norm{f}_{L^2(\R)}^2 + \norm{g}_{L^2(\R)}^2.
\]
Integration along characteristics therefore yields
\begin{align*}
  \abs{u(x,t)}^2 &\le \abs{f(x-t)}^2 + 2me^{4M} \int_0^t \abs{u}\abs{v}(x-t+s,s) \, ds,
  \\
  \abs{v(x,t)}^2 &\le \abs{g(x+t)}^2 + 2me^{4M} \int_0^t \abs{u}\abs{v}(x+t-s,s) \, ds.
\end{align*}
Estimating $2\abs{u}\abs{v} \le \abs{u}^2 + \abs{v}^2$ we then get
\begin{align*}
  \norm{u(\cdot,t)}_{D(T)}^2 &\le \norm{f}_{D(T)}^2 + me^{4M} \int_0^t \left( \norm{u(\cdot,s)}_{D(T)}^2 + \norm{v(\cdot,s)}_{D(T)}^2 \right) \, ds,
  \\
  \norm{v(\cdot,t)}_{D(T)}^2 &\le \norm{g}_{D(T)}^2 + me^{4M} \int_0^t \left( \norm{u(\cdot,s)}_{D(T)}^2 + \norm{v(\cdot,s)}_{D(T)}^2 \right) \, ds.
\end{align*}
Adding these and applying Gr\"onwall's lemma gives
\begin{equation}\label{Dbound}
  \norm{u(\cdot,t)}_{D(T)}^2 + \norm{v(\cdot,t)}_{D(T)}^2
  \le
  \left( \norm{f}_{D(T)}^2 + \norm{g}_{D(T)}^2 \right) \exp( 2me^{4M} t).
\end{equation}
Combining this with \eqref{Abound} and \eqref{Ebound} we have the all the a priori bounds needed to get global existence.

In fact, given any large target time $\tau > 0$, then by Lemma \ref{Lemma1} there exists $T_0 > 0$ such that for all $T \in (0,T_0]$,
\[
  \left( \norm{f}_{D(T)}^2 + \norm{g}_{D(T)}^2 \right) \exp( 2me^{4M} \tau )\le \varepsilon_0.
\]
Choosing $T \in (0,T_0]$ so small that
\[
  T\bigl( m + 2\norm{a_0}_{L^\infty} + 2\norm{a_1}_{L^\infty} +  2\tau\norm{E_0}_{L^\infty} +\tau M \bigr) \le \frac{\varepsilon_0}{2},
\]
we can then repeatedly apply Theorem \ref{Thm1} on successive time intervals $[0,T]$, $[0,2T]$ etc., to get existence beyond the time $\tau$.

\begin{remark}
If the coupling to the Gross-Neveu self-interaction is turned off (that is, $\lambda_3 = 0$), so that we only have the Maxwell-Dirac-Thirring equations, then the quadrilinear terms in \eqref{DelgadoFails} are not present and the proof of the a priori bound simplifies, as it is not necessary to use integrating factors. In other words, Delgado's trick works directly. The exponential factor in \eqref{Dbound} then simplifies to $\exp(2mt)$.
\end{remark}

\bibliographystyle{amsplain}
\bibliography{database}

\providecommand{\bysame}{\leavevmode\hbox to3em{\hrulefill}\thinspace}
\providecommand{\MR}{\relax\ifhmode\unskip\space\fi MR }
\providecommand{\MRhref}[2]{%
  \href{http://www.ams.org/mathscinet-getitem?mr=#1}{#2}
}
\providecommand{\href}[2]{#2}
\begin{thebibliography}{10}

\bibitem{Bachelot2006}
Alain Bachelot, \emph{Global {C}auchy problem for semilinear hyperbolic systems
  with nonlocal interactions. {A}pplications to {D}irac equations}, J. Math.
  Pures Appl. (9) \textbf{86} (2006), no.~3, 201--236. \MR{2257730}

\bibitem{Bejenaru2015}
Ioan Bejenaru and Sebastian Herr, \emph{The cubic {D}irac equation: small
  initial data in {$H^1(\Bbb{R}^3)$}}, Comm. Math. Phys. \textbf{335} (2015),
  no.~1, 43--82. \MR{3314499}

\bibitem{Bournaveas2000}
N.~Bournaveas, \emph{A new proof of global existence for the {D}irac
  {K}lein-{G}ordon equations in one space dimension}, J. Funct. Anal.
  \textbf{173} (2000), no.~1, 203--213. \MR{1760283 (2001c:35128)}

\bibitem{Bournaveas2008}
Nikolaos Bournaveas, \emph{Local well-posedness for a nonlinear {D}irac
  equation in spaces of almost critical dimension}, Discrete Contin. Dyn. Syst.
  \textbf{20} (2008), no.~3, 605--616. \MR{2373206}

\bibitem{Bournaveas2016}
Nikolaos Bournaveas and Timothy Candy, \emph{Global well-posedness for the
  massless cubic {D}irac equation}, Int. Math. Res. Not. IMRN (2016), no.~22,
  6735--6828. \MR{3632067}

\bibitem{Boussaid2016}
Nabile Boussa{\"\i}d and Andrew Comech, \emph{On spectral stability of the
  nonlinear {D}irac equation}, J. Funct. Anal. \textbf{271} (2016), no.~6,
  1462--1524. \MR{3530581}

\bibitem{Candy2011}
Timothy Candy, \emph{Global existence for an {$L^2$} critical nonlinear {D}irac
  equation in one dimension}, Adv. Differential Equations \textbf{16} (2011),
  no.~7-8, 643--666. \MR{2829499}

\bibitem{Candy2013}
\bysame, \emph{Bilinear estimates and applications to global well-posedness for
  the {D}irac-{K}lein-{G}ordon equation on {$\Bbb R^{1+1}$}}, J. Hyperbolic
  Differ. Equ. \textbf{10} (2013), no.~1, 1--35. \MR{3043488}

\bibitem{Chadam1973}
John~M. Chadam, \emph{Global solutions of the {C}auchy problem for the
  (classical) coupled {M}axwell-{D}irac equations in one space dimension}, J.
  Functional Analysis \textbf{13} (1973), 173--184. \MR{0368640 (51 \#4881)}

\bibitem{Contreras2016}
Andres Contreras, Dmitry~E. Pelinovsky, and Yusuke Shimabukuro, \emph{{$L^2$}
  orbital stability of {D}irac solitons in the massive {T}hirring model}, Comm.
  Partial Differential Equations \textbf{41} (2016), no.~2, 227--255.
  \MR{3462129}

\bibitem{Selberg2007}
Piero D'Ancona, Damiano Foschi, and Sigmund Selberg, \emph{Null structure and
  almost optimal local regularity for the {D}irac-{K}lein-{G}ordon system}, J.
  Eur. Math. Soc. (JEMS) \textbf{9} (2007), no.~4, 877--899. \MR{2341835}

\bibitem{Selberg2010b}
\bysame, \emph{Null structure and almost optimal local well-posedness of the
  {M}axwell-{D}irac system}, Amer. J. Math. \textbf{132} (2010), no.~3,
  771--839. \MR{2666908}

\bibitem{Selberg2011}
Piero D'Ancona and Sigmund Selberg, \emph{Global well-posedness of the
  {M}axwell-{D}irac system in two space dimensions}, J. Funct. Anal.
  \textbf{260} (2011), no.~8, 2300--2365. \MR{2772373}

\bibitem{Delgado1978}
V.~Delgado, \emph{Global solutions of the {C}auchy problem for the (classical)
  coupled {M}axwell-{D}irac and other nonlinear {D}irac equations in one space
  dimension}, Proc. Amer. Math. Soc. \textbf{69} (1978), no.~2, 289--296.
  \MR{0463658}

\bibitem{Dias1986}
Jo\~ao-Paulo Dias and M\'ario Figueira, \emph{Time decay for the solutions of a
  nonlinear {D}irac equation in one space dimension}, Ricerche Mat. \textbf{35}
  (1986), no.~2, 309--316. \MR{932441}

\bibitem{Grunrock2010}
Axel Gr\"unrock and Hartmut Pecher, \emph{Global solutions for the
  {D}irac-{K}lein-{G}ordon system in two space dimensions}, Comm. Partial
  Differential Equations \textbf{35} (2010), no.~1, 89--112. \MR{2748619}

\bibitem{Huh2010}
Hyungjin Huh, \emph{Global charge solutions of {M}axwell-{D}irac equations in
  {$\Bbb R^{1+1}$}}, J. Phys. A \textbf{43} (2010), no.~44, 445206, 7.
  \MR{2733825}

\bibitem{Huh2011}
\bysame, \emph{Global strong solution to the {T}hirring model in critical
  space}, J. Math. Anal. Appl. \textbf{381} (2011), no.~2, 513--520.
  \MR{2802088}

\bibitem{Huh2013}
\bysame, \emph{Global solutions to {G}ross-{N}eveu equation}, Lett. Math. Phys.
  \textbf{103} (2013), no.~8, 927--931. \MR{3063953}

\bibitem{Huh2015}
Hyungjin Huh and Bora Moon, \emph{Low regularity well-posedness for
  {G}ross-{N}eveu equations}, Commun. Pure Appl. Anal. \textbf{14} (2015),
  no.~5, 1903--1913. \MR{3359550}

\bibitem{Ikeda2013}
Masahiro Ikeda, \emph{Final state problem for the {D}irac-{K}lein-{G}ordon
  equations in two space dimensions}, Abstr. Appl. Anal. (2013), Art. ID
  273959, 11. \MR{3091223}

\bibitem{Machihara2005}
Shuji Machihara, \emph{One dimensional {D}irac equation with quadratic
  nonlinearities}, Discrete Contin. Dyn. Syst. \textbf{13} (2005), no.~2,
  277--290. \MR{2152391}

\bibitem{Machihara2007}
\bysame, \emph{Dirac equation with certain quadratic nonlinearities in one
  space dimension}, Commun. Contemp. Math. \textbf{9} (2007), no.~3, 421--435.
  \MR{2336824}

\bibitem{Machihara2010}
Shuji Machihara, Kenji Nakanishi, and Kotaro Tsugawa, \emph{Well-posedness for
  nonlinear {D}irac equations in one dimension}, Kyoto J. Math. \textbf{50}
  (2010), no.~2, 403--451. \MR{2666663}

\bibitem{Machihara2016}
Shuji Machihara and Mamoru Okamoto, \emph{Remarks on ill-posedness for the
  {D}irac-{K}lein-{G}ordon system}, Dyn. Partial Differ. Equ. \textbf{13}
  (2016), no.~3, 179--190. \MR{3521279}

\bibitem{Naumkin2016b}
I.~P. Naumkin, \emph{Cubic nonlinear {D}irac equation in a quarter plane}, J.
  Math. Anal. Appl. \textbf{434} (2016), no.~2, 1633--1664. \MR{3415743}

\bibitem{Naumkin2016}
\bysame, \emph{Initial-boundary value problem for the one dimensional
  {T}hirring model}, J. Differential Equations \textbf{261} (2016), no.~8,
  4486--4523. \MR{3537835}

\bibitem{Okamoto2013}
Mamoru Okamoto, \emph{Well-posedness and ill-posedness of the {C}auchy problem
  for the {M}axwell-{D}irac system in {$1+1$} space time dimensions}, Adv.
  Differential Equations \textbf{18} (2013), no.~1-2, 179--199. \MR{3052714}

\bibitem{Pecher2014}
Hartmut Pecher, \emph{Local well-posedness for the nonlinear {D}irac equation
  in two space dimensions}, Commun. Pure Appl. Anal. \textbf{13} (2014), no.~2,
  673--685. \MR{3117368}

\bibitem{Selberg2010}
Sigmund Selberg and Achenef Tesfahun, \emph{Low regularity well-posedness for
  some nonlinear {D}irac equations in one space dimension}, Differential
  Integral Equations \textbf{23} (2010), no.~3-4, 265--278. \MR{2588476}

\bibitem{Tesfahun2009}
Achenef Tesfahun, \emph{Global well-posedness of the 1{D}
  {D}irac-{K}lein-{G}ordon system in {S}obolev spaces of negative index}, J.
  Hyperbolic Differ. Equ. \textbf{6} (2009), no.~3, 631--661. \MR{2568812}

\bibitem{Wang2015}
Xuecheng Wang, \emph{On global existence of 3{D} charge critical
  {D}irac-{K}lein-{G}ordon system}, Int. Math. Res. Not. IMRN (2015), no.~21,
  10801--10846. \MR{3456028}

\bibitem{Zhang2014}
Aiguo You and Yongqian Zhang, \emph{Global solution to {M}axwell-{D}irac
  equations in {$1+1$} dimensions}, Nonlinear Anal. \textbf{98} (2014),
  226--236. \MR{3158454}

\bibitem{Zhang2015}
Yongqian Zhang and Qin Zhao, \emph{Global solution to nonlinear {D}irac
  equation for {G}ross-{N}eveu model in {$1+1$} dimensions}, Nonlinear Anal.
  \textbf{118} (2015), 82--96. \MR{3325607}

\end{thebibliography}

\end{document}